\documentclass[12pt]{amsart}
\usepackage{a4}
\usepackage{amssymb}

\makeatletter
\@addtoreset{equation}{section}
\makeatother

\usepackage{multirow}
\usepackage{color}
\usepackage[mathscr]{eucal}
\usepackage{amsmath,amsthm,amssymb}
\usepackage{mathrsfs}
\usepackage{enumerate}
\usepackage{slashbox}
\usepackage{bm}
\usepackage{ulem} 
\usepackage{graphicx}
\usepackage{verbatim}
\usepackage{wrapfig}
\usepackage{ascmac}
\usepackage{multicol}
\usepackage{latexsym}

\usepackage[dvipdfmx]{hyperref}

\newtheorem{Prop}{Proposition}[section]
\newtheorem{Thm}[Prop]{Theorem}
\newtheorem{Lem}[Prop]{Lemma}

\theoremstyle{definition}
\newtheorem{Def}[Prop]{Definition}
\newtheorem{Rem}[Prop]{Remark}

\newcommand{\trans}{{{}^t\!}}

\newcommand{\R}{{\mathbb R}}

\newcommand{\Z}{{\mathbb Z}}

\newcommand{\LG}{{\mathfrak g}}
\newcommand{\LH}{{\mathfrak h}}

\newcommand{\Oo}{\mathcal{O}}
\newcommand{\LLM}{{\mathfrak M}}

\newcommand{\nul}{\mathrm{null}}

\newcommand{\sign}{\mathrm{sign}}
\newcommand{\rad}{\mathrm{rad}}

\newcommand{\Aut}{\mathrm{Aut}}

\newcommand{\GL}{\mathrm{GL}}
\newcommand{\SO}{\mathrm{SO}}
\newcommand{\OO}{\mathrm{O}}

\newcommand{\s}{\mathrm{span}}

\newcommand{\RH}{\R \mathrm{H}}

\newcommand{\RAutg}{\R^\times \Aut(\LG)}

\def\hoge<#1>{\langle #1 \rangle} 

\title[A classification of left-invariant pseudo-Riemannian metrics]{A classification of left-invariant pseudo-Riemannian metrics\\ on some nilpotent Lie groups} 

\author{Yuji Kondo} 
\address[Y.~Kondo]{Department of Mathematics, Hiroshima University, 
Higashi-Hiroshima, 739-8526 Japan} 
\email{yuji-kondo@hiroshima-u.ac.jp}

\thanks{This work was partly supported by Osaka City University Advanced Mathematical Institute (MEXT Joint Usage/Research Center on Mathematics and Theoretical Physics). This work was supported by the Research Institute for Mathematical Sciences, an International Joint Usage/Research Center located in Kyoto University.}

\date{}



\begin{document} 

\maketitle

\begin{abstract}
It is known that a connected and simply-connected Lie group admits only one left-invariant Riemannian metric up to scaling and isometry if and only if it is isomorphic to the Euclidean space, the Lie group of the real hyperbolic space, or the direct product of the three dimensional Heisenberg group and the Euclidean space of dimension $n-3$. In this paper, we give a classification of left-invariant pseudo-Riemannian metrics of an arbitrary signature for the third Lie groups with $n \geq 4$ up to scaling and automorphisms. This completes the classifications of left-invariant pseudo-Riemannian metrics for the above three Lie groups up to scaling and automorphisms.
\end{abstract}


\section{Introduction}
\label{sec1}

In differential geometry, it is one of the central and fundamental problems to determine whether a given differentiable manifold admits some distinguished geometric structures or not. Such structures can be, for example, Einstein or Ricci soliton metrics (cf.\ \cite{Cao, Topping}) for the setting of Riemannian or pseudo-Riemannian manifolds, and K\"{a}hler-Einstein metrics for K\"{a}hler manifolds. When one deals with these problems, it would be natural and useful to add some other properties, such as homogeneity.

We focus on the problem whether a given Lie group admits distinguished left-invariant metrics or not, both for the Riemannian and pseudo-Riemannian cases. Left-invariant metrics on Lie groups have supplied many examples of distinguished metrics, and have been studied actively. For example, we refer to \cite{Barnet, CR, Lauret4, Lauret5, Milnor, Nomizu, Onda, Will} and references therein. In particular, we mention that the Alekseevskii's conjecture has been recently proved in \cite{BL}, which had been an open problem on homogeneous Einstein manifolds with negative scalar curvature. However, even if we consider the Riemannian cases, the present state is far from the complete.

If one can classify left-invariant metrics on a given Lie group, then it would be useful to determine the existence and non-existence of distinguished metrics. Regarding left-invariant Riemannian metrics, Lauret (\cite{Lauret}) classified connected and simply-connected Lie groups which admit only one left-invariant Riemannian metric up to scaling and isometry. Such a Lie group is isomorphic to one of
\begin{align}
\label{Lie groups with moduli=1}
\R^n, \quad G_{\RH^n} \ (n \geq 2), \quad H_3 \times \R^{n-3} \ (n \geq 3),
\end{align}
where $G_{\RH^n}$ is so-called the Lie group of the real hyperbolic space $\RH^n$ (the solvable part of the Iwasawa decomposition of the identity component $\SO^0(n, 1)$ of $\SO(n, 1)$, and acts simply-transitively on $\RH^n$), and $H_3$ is the three dimensional Heisenberg group. It is well-known that their metrics are flat, negative constant sectional curvature and Ricci soliton, respectively. For other studies on classifications of left-invariant Riemannian metrics, we refer to \cite{HL, HT, KTT, Milnor} and references therein.

We are interested in the classifications of left-invariant pseudo-Riemannian metrics on Lie groups. In the three-dimensional cases, left-invariant Lorentzian metrics have been studied in \cite{CP, Rahmani, RR}. For higher dimensional cases, it seems to be natural that we first consider the above three Lie groups, $\R^n$, $G_{\RH^n}$ and $H_3 \times \R^{n-3}$. For any signature, it is obvious that $\R^n$ admits only one left-invariant pseudo-Riemannian metric up to scaling and isometry, which is flat. For any non-Riemannian signature on $G_{\RH^n}$ ($n \geq 2$), there exist exactly three left-invariant pseudo-Riemannian metrics up to scaling and isometry, all of them have constant sectional curvatures (\cite{KOTT}). For the case of $H_3$, there exist exactly three left-invariant Lorentzian metrics up to scaling and isometry (\cite{Rahmani, RR}), and only one of them is flat and the other two are Ricci solitons but not Einstein (\cite{Nomizu, Onda, RR}).

In this paper, we consider left-invariant pseudo-Riemannian metrics on $H_3 \times \R^{n-3}$ with $n \geq 4$, and classify them up to scaling and automorphisms defined as follows.

\begin{Def}
\label{def : up to scaling and auto}
Let $g_1$ and $g_2$ be left-invariant pseudo-Riemannian metrics on a Lie group $G$. Then, $(G, g_1)$ and $(G, g_2)$ are said to be {\it equivalent up to scaling and automorphisms} if there exist $c>0$ and a Lie group automorphism $\varphi: G \to G$ such that for any $a \in G$ and $x, y \in T_aG$, they satisfy
\begin{align*}
g_1(x, y)_a=cg_2(d\varphi_a(x), d\varphi_a(y))_{\varphi(a)},
\end{align*}
where $T_aG$ is the tangent space to $G$ at $a$, and $d\varphi_a$ is the differential map of $\varphi$ at $a$.
\end{Def}

By Definition~\ref{def : up to scaling and auto}, if $(G, g_1)$ and $(G, g_2)$ are equivalent up to scaling and automorphisms, then they are isometric up to scaling. Note that the converse is not necessarily true (see Remark~\ref{rem : up to isometry}). In the preceding study \cite{KT}, it has been shown that there exist exactly six left-invariant Lorentzian metrics on $H_3 \times \R^{n-3}$ with $n \geq 4$ up to scaling and automorphisms. The main result of this paper is a classification of left-invariant pseudo-Riemannian metrics of an arbitrary signature on $H_3 \times \R^{n-3}$ with $n \geq 4$ up to scaling and automorphisms.

\begin{Thm}
\label{thm : classification of left-inv pseudo-Riem. of arbitrary signature}
Let $p, q \in \Z_{\geq 1}$ with $p+q \geq 4$. Then the number of left-invariant pseudo-Riemannian metrics of signature $(p, q)$ on $H_3 \times \R^{p+q-3}$ up to scaling and automorphisms is as follows$:$
\begin{itemize}
\item[(1)] $21$ if $p, q \geq 3$. 
\item[(2)] $15$ if $p \geq 3$ and $q=2$.
\item[(3)] $6$ if $p \geq 3$ and $q=1$.
\item[(4)] $10$ if $p=q=2$.
\end{itemize}
\end{Thm}

Note that, for any $p, q \in \Z_{\geq 0}$ and a Lie group $G$, one has the correspondence
\begin{align*}
\left\{\begin{array}{c}
		\mbox{a left-invariant metric}\\
		 \mbox{of signature $(p, q)$ on $G$}
	\end{array}\right\} \ \overset{1:1}{\longleftrightarrow} \ \left\{\begin{array}{c}
		\mbox{a left-invariant metric}\\ 
		\mbox{of signature $(q, p)$ on $G$}
	\end{array}\right\}.
\end{align*}
Therefore Theorem~\ref{thm : classification of left-inv pseudo-Riem. of arbitrary signature} gives a classification for every signature. Recall that $H_3 \times \R^{n-3}$ admits only one left-invariant Riemannian metric for $n \geq 3$, and exactly three left-invariant Lorentzian metrics for $n=3$. Combining these results with Theorem~\ref{thm : classification of left-inv pseudo-Riem. of arbitrary signature}, one has the next table of the number of left-invariant Riemannian and pseudo-Riemannian metrics of signature $(p, q)$ on $H_3 \times \R^{n-3}$. Note that our theorem completes the classifications of such metrics up to scaling and automorphisms on Lie groups in (\ref{Lie groups with moduli=1}).

\begin{table}[h]
\caption{The number of left-invariant metrics on $H_3 \times \R^{n-3}$ up to scaling and automorphisms}
\begin{center}
\begin{tabular}{|c||c|c|c|c|c|c|}\hline
\backslashbox{$p$}{$q$}&0&1&2&3&4&$\cdots$\\\hline\hline
0&&&&1&1&$\cdots$\\\hline
1&&&3&6&6&$\cdots$\\\hline
2&&3&10&15&15&$\cdots$\\\hline
3&1&6&15&21&21&$\cdots$\\\hline
4&1&6&15&21&21&$\cdots$\\\hline
$\vdots$&$\vdots$&$\vdots$&$\vdots$&$\vdots$&$\vdots$&$\ddots$\\\hline
\end{tabular}
\end{center}
\end{table}

In the proof of Theorem~\ref{thm : classification of left-inv pseudo-Riem. of arbitrary signature}, the key idea is a group action on a flag manifold. In fact, the equivalence classes of left-invariant pseudo-Riemannian metrics of signature $(p, q)$ on $H_3 \times \R^{n-3}$ up to scaling and automorphisms correspond to the orbits of the group action of the parabolic subgroup of the block decomposition $(1, n-3, 2)$
\begin{align}
\label{the form of parabolic}
\left\{\left(
		\begin{array}{c|ccc|cc}
		\ast&\ast&\cdots&\ast&\ast&\ast\\ \hline
		0&\ast&\cdots&\ast&\ast&\ast\\ 
		\vdots&\vdots&\ddots&\vdots&\vdots&\vdots\\
		0&\ast&\cdots&\ast&\ast&\ast\\ \hline
		0&0&\cdots&0&\ast&\ast\\ 
		0&0&\cdots&0&\ast&\ast
		\end{array}
	\right) \in \GL(n, \R)\right\}
\end{align}
on $\GL(n, \R)/\OO(p, q)$, which is a pseudo-Riemannian symmetric space. Moreover, this action corresponds to the action of $\OO(p, q)$ on the flag manifold given by the above parabolic subgroup. With respect to the latter action, it has been already known that the number of the orbits is finite in \cite{Wolf}. Determining the orbit space of the latter action, we classified left-invariant pseudo-Riemannian metrics of signature $(p, q)$ on $H_3 \times \R^{n-3}$ up to scaling and automorphisms. In \cite{KT}, the classification of left-invariant Lorentzian metrics on this Lie group has been obtained by matrices calculations. However if one tries to classify left-invariant non-Lorentzian metrics on it by the same method, the procedure will be very complicated, since the method depends on the signature.  In this paper, we classify left-invariant pseudo-Riemannian metrics by a method which does not depend on the signature.

We here mention the curvature properties of left-invariant pseudo-Riemannian metrics on $H_3 \times \R^{n-3}$ with $n \geq 4$. Recall that there exist exactly six left-invariant Lorentzian metrics up to scaling and automorphisms (\cite{KT}). In this case, curvatures are completely calculated, and only one of them is flat and the other five are Ricci solitons but not Einstein (\cite{KT}). For the non-Lorentzian cases, we announce that the author has partially calculated curvatures (see Remark~\ref{rem : up to isometry}), which will be in the forthcoming paper.

The author would like to thank Hiroshi Tamaru, Takayuki Okuda and Akira Kubo for the helpful discussions and suggestions. The author also expresses the sincere gratitude to Toshihiko Matsuki for the valuable comments, which had a significant influence to this study.

\section{Preliminaries}
\label{sec2}

In this section, we recall general theories on inner products on vector spaces, which are not necessarily nondegenerate, and left-invariant pseudo-Riemannian metrics on Lie groups.


\subsection{Vector spaces with inner products}

In this subsection, we recall some terminologies on vector spaces with inner products used throughout this paper, and set notations. 

First of all, let us recall the signature of an inner product. Let $V$ be an $n$-dimensional real vector space, and $\hoge< , >$ be an inner product on it, which is not necessarily nondegenerate. Fix a basis $\{v_1, \ldots, v_n\}$ of $V$, and identify $V \cong \R^n$. Then there exists a real symmetric matrix $A$ such that for any $x, y \in V$,
\begin{align*}
\hoge<x, y>=\trans xAy.
\end{align*}
Since $A$ is a real symmetric matrix, every eigenvalue of $A$ is a real number. Note that $0$ can be its eigenvalue since $\hoge< , >$ is not necessarily nondegenerate. Then the triplet of the numbers of positive, negative and zero eigenvalues of $A$ counted with multiplicities is called the {\it signature} of $\hoge< , >$ on $V$, and we denote it by
\begin{align*}
\sign (V, \hoge< , >)=(p, q, r) \quad (p, q, r \in \Z_{\geq 0}).
\end{align*}
In the cases of $r=0$, that is, when $\hoge< , >$ is nondegenerate on $V$, we may write $\sign (V, \hoge< , >)=(p, q)$. If we do not need to specify $\hoge< , >$, we denote it by $\sign V$ for simplicity. We use this notation for a subspace $W$ of $V$ as well, that is, we denote the signature of $\hoge< , >\mid_{W \times W}$ on $W$ by
\begin{align*}
\sign (W, \hoge< , >\mid_{W \times W})=(s, t, u) \quad (s, t, u \in \Z_{\geq 0}).
\end{align*}
In this case, we write $\sign (W, \hoge< , >)$ or $\sign W$ for simplicity.

Next we recall the radical. The \textit{radical} $\rad (V, \hoge< , >)$ of $V$ with respect to $\hoge< , >$ is a subspace of $V$ such that its vector is orthogonal to every vector of $V$, that is,
\begin{align*}
\rad (V, \hoge< , >):=\{v \in V \mid \forall w \in V,\ \hoge<v, w>=0\}.
\end{align*}
Similarly, we may write $\rad V$ for simplicity. For a subspace $W$ of $V$, we simply denote the radical of $W$ with respect to $\hoge< , >\mid_{W \times W}$ by $\rad (W, \hoge< , >)$ or $\rad W$.

\subsection{The spaces of left-invariant pseudo-Riemannian metrics on Lie groups}

In this subsection, we recall the notion of the spaces of left-invariant pseudo-Riemannian metrics on Lie groups. This has been introduced in \cite{KOTT}. We refer to \cite{KTT} for the Riemannian case. In the following arguments, let $G$ be a real Lie group of dimension $n$, and $\LG$ be the corresponding Lie algebra. We fix a basis $\{e_1, \ldots, e_n\}$ of $\LG$, and identify $\LG \cong \R^n$ as vector spaces.

Let $p, q \in \Z_{\geq 1}$. Recall that a pseudo-Riemannian metric has signature $(p, q)$ if so is the induced inner product on each tangent space. We are interested in a classification of left-invariant pseudo-Riemannian metrics on $G$. For this purpose, we denote the space of left-invariant pseudo-Riemannian metrics by 
\begin{align*}
\LLM_{(p, q)}(G):=\{\mbox{a\ left-invariant\ metric\ of\ signature}\ (p, q)\ \mbox{on\ }G\}.
\end{align*}
We then consider the counterpart in the Lie algebra $\LG$ of $G$, and denote it by
\begin{align*}
\LLM_{(p, q)}(\LG):=\{\hoge< , > : \mbox{an inner product of signature}\ (p, q)\ \mbox{on}\ \LG\}.
\end{align*}
It is well-known that there exists a one-to-one correspondence between $\LLM_{(p, q)}(G)$ and $\LLM_{(p, q)}(\LG)$. Recall that we identify $\LG \cong \R^n$. Then $\GL(n, \R)$ acts transitively on $\LLM_{(p, q)}(\LG)$ by
\begin{align*}
g.\hoge<x, y>:=\hoge<g^{-1}x, g^{-1}y> \quad (\forall g \in \GL(n, \R),\ \forall x, y \in \LG).
\end{align*}

From now on, we explain the equivalence relation on inner products, which corresponds to the equivalence relation on $\LLM_{(p, q)}(G)$ given by Definition~\ref{def : up to scaling and auto}. Let us consider the automorphism group of $\LG$,
\begin{align*}
\Aut(\LG):=\{\varphi \in \GL(n, \R) \mid \forall x, y \in \LG,\ \varphi([x, y])=[\varphi(x), \varphi(y)]\},
\end{align*}
and also put $\R^\times :=\R \setminus \{0\}$. We study the group action by
\begin{align*}
\R^\times \Aut(\LG):=\{c\varphi \in \GL(n, \R) \mid c \in \R^\times,\ \varphi \in \Aut(\LG)\}.
\end{align*}
This is a subgroup of $\GL(n, \R)$, and thus it naturally acts on $\LLM_{(p, q)}(\LG)$. We denote the orbit through $\hoge< , >$ by $\RAutg.\hoge< , >$. 

\begin{Def}
\label{def : isometric up to scaling}
Let $\hoge< , >_1, \hoge< , >_2 \in \LLM_{(p, q)}(\LG)$. Then, $(\LG, \hoge< , >_1)$ and $(\LG, \hoge< , >_2)$ are said to be {\it equivalent up to scaling and automorphisms} if they satisfy
\begin{align*}
\hoge< , >_1 \in \R^\times \Aut(\LG).\hoge< , >_2.
\end{align*}
\end{Def}

This notion is an equivalence relation on $\LLM_{(p, q)}(\LG)$. If a given Lie group $G$ is connected and simply-connected, then one knows $\Aut(G) \cong \Aut(\LG)$, and therefore the classification of inner products on $\LG$ by the action of $\RAutg$ is equivalent to the classification of left-invariant pseudo-Riemannian metrics on $G$ up to scaling and automorphisms. Hence it is natural to study the following orbit space:
\begin{align*}
\R^\times \Aut(\LG) \backslash \LLM_{(p, q)}(\LG):=\{\R^\times \Aut(\LG).\hoge< , > \mid \hoge< , > \in \LLM_{(p, q)}(\LG)\}.
\end{align*}
This space can be regarded as the moduli space of left-invariant pseudo-Riemannian metrics on $G$ of signature $(p, q)$.

Finally in this subsection, we give a remark on a classification of left-invariant pseudo-Riemannian metrics on $G$ up to scaling and isometry, which is defined as follows.

\begin{Def}
\label{def : isometric up to scaling on Lie group}
Let $g_1, g_2 \in \LLM_{(p, q)}(G)$. Then, $(G, g_1)$ and $(G, g_2)$ are said to be {\it isometric up to scaling} and denoted by $g_1 \sim_G g_2$ if there exist $c>0$ and a diffeomorphism $\varphi: G \to G$ such that for any $a \in G$ and $x, y \in T_aG$,
\begin{align*}
g_1(x, y)_a=cg_2(d\varphi_a(x), d\varphi_a(y))_{\varphi(a)}.
\end{align*}
\end{Def}

One can define an equivalence relation $\sim_{\LG}$ on $\LLM_{(p, q)}(\LG)$ corresponding to $\sim_G$, that is, there exists a one-to-one correspondence
\begin{align*}
\LLM_{(p, q)}(G)/\sim_G \ \overset{1:1}{\longleftrightarrow} \ \LLM_{(p, q)}(\LG)/\sim_\LG.
\end{align*}
By Definition~\ref{def : up to scaling and auto}, if two left-invariant metrics are equivalent up to scaling and automorphisms, then they are isometric up to scaling. Thus there exists a surjection
\begin{align*}
\RAutg \backslash \LLM_{(p, q)}(\LG) \twoheadrightarrow \LLM_{(p, q)}(\LG)/\sim_\LG.
\end{align*}
In this paper, as mentioned above, we focus on the classification of inner products by the action of $\RAutg$. In order to obtain the classification up to $\sim_G$ or $\sim_{\LG}$, we need to distinguish elements in $\RAutg \backslash \LLM_{(p, q)}(\LG)$, which can be equivalent in the sense of $\sim_\LG$.

\section{An outline of the proof of the main theorem}
\label{sec3}

In this section, we describe an outline of the proof of Theorem~\ref{thm : classification of left-inv pseudo-Riem. of arbitrary signature}, which can be divided into some parts. In the first subsection, we consider the orbit-decomposition with respect to the action of an indefinite orthogonal group on a flag manifold. In the second subsection, we describe possible signatures on particular vector subspaces. We explain the statements of them without giving proofs, and we prove Theorem~\ref{thm : classification of left-inv pseudo-Riem. of arbitrary signature} in the last subsection.

Let $I_k$ be the unit matrix of order $k$, and put
\begin{align*}
I_{p, q}:=\left(
\begin{array}{cc}
I_p&\\
&-I_q
\end{array}
\right),
\end{align*}
where $p, q \in \Z_{\geq 1}$. We consider the standard inner product $\hoge< , >_0$ such that $\sign (\R^{p+q}, \hoge< , >_0)=(p, q)$, that is, it is defined by
\begin{align*}
\hoge<x, y>_0:=\trans xI_{p, q}y \quad (\forall x, y \in \R^{p+q}).
\end{align*}

\subsection{An orbit-decomposition of a flag manifold}

In this subsection, we describe the orbit-decomposition with respect to the action of the indefinite orthogonal group $\OO(p, q)$ on the flag manifold
\begin{align*}
F_{k_1, k_2}:=\{(V_{k_1}, V_{k_2}) \mid V_{k_1} \subset V_{k_2} \subset \R^{p+q},\ \dim V_{k_i}=k_i\ (i=1, 2)\},
\end{align*}
where $k_1, k_2 \in \{1, \ldots, p+q\}$ with $k_1<k_2$. Note that $\OO(p, q)$ acts on $F_{k_1, k_2}$ by
\begin{align*}
g.(V_{k_1}, V_{k_2}):=(gV_{k_1}, gV_{k_2}).
\end{align*}
For flags in $F_{k_1, k_2}$, an equivalent condition to be contained in the same $\OO(p, q)$-orbit is given in terms of the signatures as follows.

\begin{Prop}
\label{prop : flag mfd iff signature and radical}
For any $(V_{k_1}, V_{k_2}), (W_{k_1}, W_{k_2}) \in F_{k_1, k_2}$, the following conditions are equivalent.
\begin{itemize}
\item[(1)] There exists $g \in \OO(p, q)$ such that $(V_{k_1}, V_{k_2})=g.(W_{k_1}, W_{k_2})$.
\item[(2)] All of the following hold. 
	\begin{itemize}
	\item[(i)] $\sign (V_{k_2}, \hoge< , >_0)=\sign (W_{k_2}, \hoge< , >_0)$.
	\item[(ii)] $\sign (V_{k_1}, \hoge< , >_0)=\sign (W_{k_1}, \hoge< , >_0)$.
	\item[(iii)] $\dim (V_{k_1} \cap \rad (V_{k_2}, \hoge< , >_0))=\dim (W_{k_1} \cap \rad (W_{k_2}, \hoge< , >_0))$.
	\end{itemize}
\end{itemize}
\end{Prop}

We will give the proof of this proposition in Section~\ref{sec4}. From this proposition, each $\OO(p, q)$-orbit through $(V_{k_1}, V_{k_2}) \in F_{k_1, k_2}$ is characterized only by the three data
\begin{align*}
\sign V_{k_1}, \quad \sign V_{k_2}, \quad \dim (V_{k_1} \cap \rad V_{k_2}).
\end{align*}

\begin{Rem}
\label{rem : Matsuki duality}
For a reductive affine symmetric space $(G, H, \sigma)$ and its associated affine symmetric space $(G, H', \sigma \theta)$, Matsuki (\cite{Matsuki, Matsuki1}) showed the correspondence between the double cosets, that is, one has 
\begin{align*}
H \backslash G/P\ \overset{1:1}{\longleftrightarrow} \ H' \backslash G/P,
\end{align*}
where $P$ is a parabolic subgroup of $G$. This correspondence is called the \textit{Matsuki duality} (\textit{correspondence}). We consider $\RAutg$ for $\LG:=\LH_3 \oplus \R^{n-3}$ with $n \geq 4$, which is given in the form of (\ref{the form of parabolic}). Therefore in the case of this paper, we put
\begin{align}
\label{setting of Matsuki duality in this paper}
G:=\GL(n, \R), \quad H:=\OO(p, q), \quad P:=\RAutg.
\end{align}
Then one has $H'=\GL(p, \R) \times \GL(q, \R)$, and hence we have the setting of the Matsuki duality. 

On the other hand, one can also determine the orbit space of the action of $H'$ on $G/P=F_{1, n-2}$ in the setting (\ref{setting of Matsuki duality in this paper}). For this purpose, we put
\begin{align*}
U^+:=\s\{e_1, \ldots, e_p\}, \quad U^-:=\s\{e_{p+1}, \ldots, e_{p+q}\},
\end{align*}
where $\{e_1, \ldots, e_{p+q}\}$ is the standard basis of $\R^{p+q}$, and consider the following data for any $(V_1, V_{n-2}) \in F_{1, n-2}$:
\begin{align*}
&c^+:=\dim (V_{n-2} \cap U^+), \quad c^-:=\dim (V_{n-2} \cap U^-), \quad c^0:=n-2-c^+-c^-,\\
&d^+:=\dim (V_1 \cap U^+), \quad d^-:=\dim (V_1 \cap U^-), \quad d^0:=1-d^+ -d^-,\\
&d^\pm:=\dim \left(((V_{n-2} \cap U^+) \oplus (V_{n-2} \cap U^-)) \cap V_1\right).
\end{align*}
Then every orbit of $H'$ on $G/P=F_{1, n-2}$ is determined by the above seven data. This fact is essentially the same as Proposition~\ref{prop : flag mfd iff signature and radical} in the case of $k_1=1$ and $k_2=n-2$.

\end{Rem}

\subsection{Possible signatures on some subspaces}

Recall that the $\OO(p, q)$-orbit through $(V_{k_1}, V_{k_2}) \in F_{k_1, k_2}$ is characterized by the three data. In this subsection, we here describe all possible three data for the case $k_1=1$ and $k_2=p+q-2$. The next proposition describes all possible $\sign V_{p+q-2}$.

\begin{Prop}
\label{prop : signature on V_{p+q-2}}
Let $A$ be the set of all possible signatures of codimension-two subspaces of $\R^{p+q}$ with respect to $\hoge< , >_0$, that is,
\begin{align*}
A:=\{\sign (V, \hoge< , >_0) \mid V \subset \R^{p+q},\ \dim V=p+q-2\}.
\end{align*}
Then one has 
\begin{align*}
A=\left\{
\begin{array}{lll}
	(p-2, q, 0), &(p-1, q-1, 0), &(p, q-2, 0),\\
	(p-2, q-1, 1), &(p-1, q-2, 1), &(p-2, q-2, 2)
\end{array}\right\} \cap (\Z_{\geq 0})^3.
\end{align*}
\end{Prop}

The proof will be given in Section~\ref{sec4}.

Fix a subspace $V$ of $\R^{p+q}$ with $\dim V \geq 2$. Take an arbitrary one dimensional subspace $W$ of $V$. Then one has 
\begin{align*}
\sign W \in \{(1, 0, 0),\ (0, 1, 0),\ (0, 0, 1)\}.
\end{align*}
According to Proposition~\ref{prop : flag mfd iff signature and radical}, when $\sign W=(0, 0, 1)$, we need to know $\dim (W \cap \rad V)$. Hence we define the new notion 
\begin{align*}
\sign_V (W, \hoge< , >_0):=
\begin{cases}
\sign (W, \hoge< , >_0)&\mbox{if $W \cap \rad (V, \hoge< , >_0)=\{0\}$},\\
(0, 0, 1)_\nul&\mbox{if $W \subset \rad (V, \hoge< , >_0)$}.
\end{cases}
\end{align*}
It is obvious that one has
\begin{align*}
\sign_V W \in \{(1, 0, 0),\ (0, 1, 0),\ (0, 0, 1),\ (0, 0, 1)_\nul\}.
\end{align*}
The next proposition describes all possible $\sign_V W$. 

\begin{Prop}
\label{prop : sign_V W}
Fix a subspace $V$ of $\R^{p+q}$ with $\sign (V, \hoge< , >_0)=(s, t, u)$ and $s, t, u \in \Z_{\geq 0}$. Let $B$ be the set of all possible signatures of one-dimensional subspaces of $V$ with respect to $\hoge< , >_0$, that is,
\begin{align*}
B:=\{\sign_V (W, \hoge< , >_0) \mid W \subset V,\ \dim W=1\}.
\end{align*}
Then one has
\begin{itemize}
	\setlength{\itemsep}{3pt} 
\item $(1, 0, 0) \in B$ if and only if $s \geq 1$,
\item $(0, 1, 0) \in B$ if and only if $t \geq 1$,
\item $(0, 0, 1) \in B$ if and only if $s, t \geq 1$,
\item $(0, 0, 1)_\nul \in B$ if and only if $u \geq 1$.
\end{itemize}
\end{Prop}

Also for this proposition, the proof will be given in Section~\ref{sec4}.

\subsection{The proof of the main theorem}

In this subsection, we prove Theorem~\ref{thm : classification of left-inv pseudo-Riem. of arbitrary signature} by applying Propositions~\ref{prop : flag mfd iff signature and radical}, \ref{prop : signature on V_{p+q-2}}, and \ref{prop : sign_V W}. Let $G:=H_3 \times \R^{n-3}$ with $n \geq 4$ and 
\begin{align*}
\LG:=\LH_3 \oplus \R^{n-3}:=\s\{e_1, \ldots, e_n \mid [e_{n-1}, e_n]=e_1\},
\end{align*}
where $\LH_3=\s\{e_1, e_{n-1}, e_n\}$ is the three dimensional Heisenberg Lie algebra. 

\begin{proof}[Proof of Theorem~\ref{thm : classification of left-inv pseudo-Riem. of arbitrary signature}]
Let $p, q \in \Z_{\geq 1}$ with $p+q \geq 4$. The desired classification is given by the orbits of the action of $\RAutg$ on $\LLM_{(p, q)}(\LG)$. Recall that one has an identification
\begin{align*}
\LLM_{(p, q)}(\LG)=\GL(n, \R)/\OO(p, q)
\end{align*}
as homogeneous spaces, where $n=p+q$. Hence, we can identify the orbit space $\RAutg \backslash \LLM_{(p, q)}(\LG)$ with the double coset space, that is, one has
\begin{align*}
\RAutg \backslash \LLM_{(p, q)}(\LG)=\RAutg \backslash \GL(n, \R)/\OO(p, q).
\end{align*}
On the other hand, from a general theory, there is a one-to-one correspondence 
\begin{align*}
\RAutg \backslash \GL(n, \R)/\OO(p, q)\ \overset{1:1}{\longleftrightarrow}\ \OO(p, q) \backslash \GL(n, \R)/\RAutg.
\end{align*}
Moreover the matrix expression of $\RAutg$ with respect to the basis $\{e_1, \dots, e_n\}$ of $\LG$ coincides with the form of (\ref{the form of parabolic}) (cf.\ \cite{KTT}). By this matrix expression, $\GL(n, \R)/\RAutg$ can be identified with the flag manifold $F_{1, p+q-2}$. From the above arguments, $\RAutg \backslash \LLM_{(p, q)}(\LG)$ corresponds to $\OO(p, q) \backslash F_{1, p+q-2}$. Therefore we have only to classify flags in $F_{1, p+q-2}$ by the action of $\OO(p, q)$. By Proposition~\ref{prop : flag mfd iff signature and radical}, one knows that each $\OO(p, q)$-orbit through $(V_1, V_{p+q-2}) \in F_{1, p+q-2}$ is characterized only by 
\begin{align*}
\sign V_1, \quad \sign V_{p+q-2}, \quad \dim (V_1 \cap \rad V_{p+q-2}).
\end{align*}

In the following arguments, we assume $p \geq q$. Then the condition $p+q \geq 4$ yields that $p \geq 2$. From Proposition~\ref{prop : signature on V_{p+q-2}}, one has
\begin{align*}
\sign V_{p+q-2} \in \left\{
\begin{array}{lll}
	(p-2, q, 0), &(p-1, q-1, 0), &(p, q-2, 0),\\
	(p-2, q-1, 1), &(p-1, q-2, 1), &(p-2, q-2, 2)
\end{array}\right\} \cap (\Z_{\geq 0})^3.
\end{align*} 
We complete Table~\ref{table : The number of equivalence classes} for each $\sign V_{p+q-2}$ one by one. 
Note that
\begin{align*}
p \geq 2, \quad q \geq 1.
\end{align*}
First, let us consider the case of $\sign V_{p+q-2}=(p-2, q, 0)$. In this case, by Proposition~\ref{prop : sign_V W}, 
\begin{itemize}
	\setlength{\itemsep}{7pt} 
\item if $p \geq 3$, then $\sign_{V_{p+q-2}} V_1=(1, 0, 0),\ (0, 1, 0),\ (0, 0, 1)$,
\item if $p=2$, then $\sign_{V_{p+q-2}} V_1=(0, 1, 0)$.
\end{itemize}
We here summarize all possible $\sign_{V_{p+q-2}} V_1$ for the other $\sign V_{p+q-2}$. In the case of $\sign V_{p+q-2}=(p-1, q-1, 0)$,
\begin{itemize}
	\setlength{\itemsep}{7pt} 
\item if $q \geq 2$, then $\sign_{V_{p+q-2}} V_1=(1, 0, 0),\ (0, 1, 0),\ (0, 0, 1)$,
\item if $q=1$, then $\sign_{V_{p+q-2}} V_1=(1, 0, 0)$.
\end{itemize}
In the case of $\sign V_{p+q-2}=(p, q-2, 0)$, we have $q \geq 2$ and
\begin{itemize}
	\setlength{\itemsep}{7pt} 
\item if $q \geq 3$, then $\sign_{V_{p+q-2}} V_1=(1, 0, 0),\ (0, 1, 0),\ (0, 0, 1)$,
\item if $q=2$, then $\sign_{V_{p+q-2}} V_1=(1, 0, 0)$.
\end{itemize}
In the case of $\sign V_{p+q-2}=(p-2, q-1, 1)$,
\begin{itemize}
	\setlength{\itemsep}{7pt} 
\item if $p \geq 3$ and $q \geq 2$, then $\sign_{V_{p+q-2}} V_1=(1, 0, 0),\ (0, 1, 0),\ (0, 0, 1),\ (0, 0, 1)_\nul$,
\item if $p \geq 3$ and $q=1$, then $\sign_{V_{p+q-2}} V_1=(1, 0, 0),\ (0, 0, 1)_\nul$,
\item if $p=q=2$, then $\sign_{V_{p+q-2}} V_1=(0, 1, 0),\ (0, 0, 1)_\nul$.
\end{itemize}
In the case of $\sign V_{p+q-2}=(p-1, q-2, 1)$, we have $q \geq 2$ and
\begin{itemize}
	\setlength{\itemsep}{7pt} 
\item if $p \geq 3$ and $q \geq 3$, then $\sign_{V_{p+q-2}} V_1=(1, 0, 0),\ (0, 1, 0),\ (0, 0, 1),\ (0, 0, 1)_\nul$,
\item if $p \geq 3$ and $q=2$, then $\sign_{V_{p+q-2}} V_1=(1, 0, 0),\ (0, 0, 1)_\nul$,
\item if $p=q=2$, then $\sign_{V_{p+q-2}} V_1=(1, 0, 0),\ (0, 0, 1)_\nul$.
\end{itemize}
In the case of $\sign V_{p+q-2}=(p-2, q-2, 2)$, we have $q \geq 2$ and
\begin{itemize}
	\setlength{\itemsep}{7pt} 
\item if $p \geq 3$ and $q \geq 3$, then $\sign_{V_{p+q-2}} V_1=(1, 0, 0),\ (0, 1, 0),\ (0, 0, 1),\ (0, 0, 1)_\nul$,
\item if $p \geq 3$ and $q=2$, then $\sign_{V_{p+q-2}} V_1=(1, 0, 0),\ (0, 0, 1)_\nul$,
\item if $p=q=2$, then $\sign_{V_{p+q-2}} V_1=(0, 0, 1)_\nul$.
\end{itemize}

Hence one can obtain the pairs of $\sign V_{p+q-2}$ and $\sign_{V_{p+q-2}} V_1$ in Table~\ref{table : The number of equivalence classes}. Only for the case of $p, q \geq 3$, we explicitly describe $21$ pairs of the signatures, and for the other cases we mark each slot in the table with the check mark ``\checkmark" if its corresponding equivalence class appears. At the bottom row, we write the number of equivalence classes.

\begin{table}[h]
\caption{The number of equivalence classes}
	\begin{tabular}{|c||l|l||c||c||c|}\hline
	&\multicolumn{2}{|c||}{$p, q \geq 3$}&$p \geq 3$, $q=2$&$p \geq 3$, $q=1$&$p=q=2$\\ \hline
	&$\sign V_{p+q-2}$&$\sign_{V_{p+q-2}} V_1$&&&\\\hline\hline
	(1)&$(p-2, q, 0)$&$(1, 0, 0)$&\checkmark&\checkmark&\\
	(2)&&$(0, 1, 0)$&\checkmark&\checkmark&\checkmark\\
	(3)&&$(0, 0, 1)$&\checkmark&\checkmark&\\\hline
	(4)&$(p-1, q-1, 0)$&$(1, 0, 0)$&\checkmark&\checkmark&\checkmark\\
	(5)&&$(0, 1, 0)$&\checkmark&&\checkmark\\
	(6)&&$(0, 0, 1)$&\checkmark&&\checkmark\\\hline
	(7)&$(p, q-2, 0)$&$(1, 0, 0)$&\checkmark&&\checkmark\\
	(8)&&$(0, 1, 0)$&&&\\
	(9)&&$(0, 0, 1)$&&&\\\hline
	(10)&$(p-2, q-1, 1)$&$(1, 0, 0)$&\checkmark&\checkmark&\\
	(11)&&$(0, 1, 0)$&\checkmark&&\checkmark\\
	(12)&&$(0, 0, 1)$&\checkmark&&\\
	(13)&&$(0, 0, 1)_\nul$&\checkmark&\checkmark&\checkmark\\\hline
	(14)&$(p-1, q-2, 1)$&$(1, 0, 0)$&\checkmark&&\checkmark\\
	(15)&&$(0, 1, 0)$&&&\\
	(16)&&$(0, 0, 1)$&&&\\
	(17)&&$(0, 0, 1)_\nul$&\checkmark&&\checkmark\\\hline
	(18)&$(p-2, q-2, 2)$&$(1, 0, 0)$&\checkmark&&\\
	(19)&&$(0, 1, 0)$&&&\\
	(20)&&$(0, 0, 1)$&&&\\
	(21)&&$(0, 0, 1)_\nul$&\checkmark&&\checkmark\\\hline\hline
	&\multicolumn{2}{|c||}{$21$}&$15$&$6$&$10$\\ \hline
	\end{tabular}
\label{table : The number of equivalence classes}
\end{table}
This table proves Theorem~\ref{thm : classification of left-inv pseudo-Riem. of arbitrary signature}. 
\end{proof}

For $p, q \in \Z_{\geq 1}$ with $p+q \geq 4$, every $\OO(p, q)$-orbit in $F_{1, p+q-2}$ is characterized by $\sign (V_{p+q-2}, \hoge< , >_0)$ and $\sign_{V_{p+q-2}} (V_1, \hoge< , >_0)$ as in Table~\ref{table : The number of equivalence classes}. We explain what this table represents in terms of inner products on $\LG$. Here we denote the center and the derived ideal of $\LG$ by $Z(\LG)$ and $[\LG, \LG]$, respectively. Then one has
\begin{align*}
Z(\LG)=\s\{e_1, \ldots, e_{p+q-2}\}, \quad [\LG, \LG]=\s\{e_1\}.
\end{align*}
In terms of $\LG$, Table~\ref{table : The number of equivalence classes} represents the pairs of signatures of $\hoge< , > \in \LLM_{(p, q)}(\LG)$ restricted to $Z(\LG)$ and $[\LG, \LG]$, that is, every $\RAutg$-orbit in $\LLM_{(p, q)}(\LG)$ is characterized by 
\begin{align*}
\sign (Z(\LG), \hoge< , >), \quad \sign_{Z(\LG)} ([\LG, \LG], \hoge< , >)
\end{align*}
as in Table~\ref{table : The number of equivalence classes}.

\begin{Rem}
\label{rem : flat metrics}
For left-invariant Lorentzian metrics on $G$, the degenerations of $\RAutg$-orbits have been studied in \cite{KT}. For any different orbits $\Oo_1$ and $\Oo_2$, recall that $\Oo_1$ is said to {\it degenerate} to $\Oo_2$ if $\Oo_2 \subset \overline{\Oo_1}$ holds, where $\overline{\Oo_1}$ is the closure of $\Oo_1$. In the Lorentzian case, there exists only one closed $\RAutg$-orbit, which corresponds to (13) in Table~\ref{table : The number of equivalence classes} and is characterized as the unique equivalence class of flat metrics up to scaling and automorphisms. Furthermore, inner products in this closed orbit are degenerate on $Z(\LG)$ and $[\LG, \LG]$ as (13) in Table~\ref{table : The number of equivalence classes}. The author has verified that similar phenomena occur also in the non-Lorentzian cases, that is, 
\begin{itemize}
\item the $\RAutg$-orbit corresponding to (21) is the unique closed orbit,
\item the metric corresponding to (21) is flat,
\item inner products in this closed orbit are degenerate on $Z(\LG)$ and $[\LG, \LG]$ as (21) in Table~\ref{table : The number of equivalence classes}.
\end{itemize}
Note that a closed orbit always exists. It would be a natural problem to consider whether the above three correspondences hold for any Lie group or not. In fact, some papers study the relations between the curvature properties and the signatures of the restrictions (\cite{BT, GB}).
\end{Rem}

\begin{Rem}
\label{rem : up to isometry}
For a fixed signature, we here mention that the left-invariant pseudo-Riemannian metrics on $G$ corresponding to $(13)$, $(17)$, $(20)$ and $(21)$ in Table~\ref{table : The number of equivalence classes} are all isometric to each other. The curvatures of the above metrics can be calculated directly. According to it, the left-invariant pseudo-Riemannian metrics on $G$ corresponding to $(13)$, $(17)$, $(20)$ and $(21)$ in Table~\ref{table : The number of equivalence classes} are flat metrics. In \cite{G}, it is proved that every left-invariant pseudo-Riemannian metric on a two-step nilpotent Lie group is geodesically complete. Hence $G$ endowed with one of the above four flat metrics is a simply-connected space form, where a \textit{space form} is a complete and connected pseudo-Riemannian manifold with constant curvature. It is well-known that simply-connected space forms are isometric if and only if they have the same dimension, signature and constant curvature (cf.\ \cite{O'Neill}). Therefore, our claim holds.

Recall that the metrics corresponding to $(17)$, $(20)$ and $(21)$ occur only in the non-Lorentzian cases. Thus in the non-Lorentzian cases, there exist left-invariant pseudo-Riemannian metrics on $G$ which are distinct up to automorphisms but isometric.
\end{Rem}

\section{The proof of Propositions~\ref{prop : flag mfd iff signature and radical}, \ref{prop : signature on V_{p+q-2}} and \ref{prop : sign_V W}}
\label{sec4}

In this section, we prove the propositions which we used for proving the main theorem in Section~\ref{sec3}. Throughout this section, let $V$ be a real vector space of finite dimension. We denote by $\hoge< , >$ an inner product on $V$, which is not necessarily nondegenerate.

\subsection{Auxiliary lemmas and propositions on vector spaces}
\label{sub1}

In this subsection, we show some auxiliary lemmas and propositions, which we use in Subsections~\ref{sub2} and \ref{sub3}.  

First of all, we define a particular basis for a given vector space, which is an analogue to an orthonormal basis in the positive definite case. In order to do that, we introduce the next notation $\varepsilon_i$ given by
\begin{align*}
\varepsilon_i:=\left\{
			\begin{array}{rl}
				1&(i \in \{1, \ldots, p\}),\\
				-1&(i \in \{s+1, \ldots, p+q\}),\\
				0&(i \in \{s+t+1, \ldots, p+q+r\}),
			\end{array}
			\right.
\end{align*}
where $p, q, r \in \Z_{\geq 0}$. 

\begin{Def}
\label{def : (p, q, r)-basis}
A set $\{v_1, \ldots, v_{p+q+r}\}$ of linearly independent vectors of $V$ is called a \textit{$(p, q, r)$-system} with respect to $\hoge< , >$ if it satisfies
\begin{align*}
\hoge<v_i, v_j>=\varepsilon_i \delta_{ij} \quad (\forall i, j \in \{1, \ldots, p+q+r\}),
\end{align*}
where $\delta_{ij}$ is the Kronecker's delta. In addition, if $\{v_1, \ldots, v_{p+q+r}\}$ is a basis of $V$, then it is called a \textit{$(p, q, r)$-basis} of $V$. 
\end{Def}

A vector space of finite dimension with a positive definite inner product has an orthonormal basis. A similar statement holds for the nondegenerate cases (cf.\ \cite{DB}). More generally, there exists a $(p, q, r)$-basis of $V$ if $\sign V=(p, q, r)$.

\begin{Prop}
\label{prop : existence of (p, q, r)-basis}
Let $(p, q, r):=\sign (V, \hoge< , >)$. Then $V$ has a $(p, q, r)$-basis with respect to $\hoge< , >$.
\end{Prop}

\begin{proof}
We identify $V \cong \R^{p+q+r}$ as vector spaces. Let $\{e_1, \ldots, e_{p+q+r}\}$ be the standard basis of $V$, and we put
\begin{align*}
I_{p, q, r}:=\left(
\begin{array}{ccc}
I_p&&\\
&-I_q&\\
&&O_r
\end{array}
\right),
\end{align*}
where $O_r$ is the zero matrix of order $r$. Let $A$ be the Gram matrix of $\hoge< , >$ with respect to $\{e_1, \ldots, e_{p+q+r}\}$. Then by Sylvester's law of inertia, there exists $g \in \GL(p+q+r, \R)$ such that $\trans gAg=I_{p, q, r}$. Here we put
\begin{align*}
v_i:=ge_i \quad (i \in \{1, \ldots, p+q+r\}).
\end{align*}
One obtains a $(p, q, r)$-basis $\{v_1, \ldots, v_{p+q+r}\}$ of $V$ with respect to $\hoge< , >$.
\end{proof}

Next we consider the decomposition of a light-like vector $v \notin \rad V$ into space-like and time-like vectors. Recall that a vector $v \in V$ is
\begin{itemize}
\item \textit{space-like} if $\hoge<v, v> >0$ or $v=0$,
\item \textit{time-like} if $\hoge<v, v> <0$,
\item \textit{light-like} if $\hoge<v, v>=0$ and $v \neq 0$.
\end{itemize}
Let $U$ be a nondegenerate subspace of $V$ with respect to $\hoge< , >$, and define the \textit{light-cone} of $U$ by
\begin{align*}
C_0(U, \hoge< , >):=\{u \in U \mid \hoge<u, u>=0\} \setminus \{0\}.
\end{align*}
Moreover we put
\begin{align*}
\OO(U, \hoge< , >):=\{f: U \to U \mid \mbox{$f$ is a linear isometry with respect to $\hoge< , >$}\}.
\end{align*}
Then it is well-known that $C_0(U, \hoge< , >)$ is an $\OO(U, \hoge< , >)$-homogeneous space.

\begin{Lem}
\label{prop : decompose to space and time}
Let $v$ be a light-like vector in $V$ with $v \notin \rad V$. Then there exists a $(1, 1, 0)$-system $\{v^+, v^-\}$ of $V$ such that $v=v^+ +v^-$.
\end{Lem}

\begin{proof}
Since $v \notin \rad V$, there exists a subspace $U$ of $V$ such that
\begin{align*}
V=U \oplus \rad V, \quad v \in U.
\end{align*}
Then there exist $p, q \in \Z_{\geq 1}$ such that $\sign U=(p, q)$ with respect to $\hoge< , >$, since $v \in U$ is light-like. Hence $U$ contains $e^+$ and $e^-$ such that 
\begin{align*}
\hoge<e^+, e^+>=1, \quad \hoge<e^-, e^->=-1, \quad \hoge<e^+, e^->=0.
\end{align*}
Thus one has $e^++e^- \in C_0(U, \hoge< , >)$. Since $C_0(U, \hoge< , >)$ is an $\OO(U, \hoge< , >)$-homogeneous space, there exists $f \in \OO(U, \hoge< , >)$ such that
\begin{align*}
v=f(e^++e^-)=f(e^+)+f(e^-).
\end{align*}
By putting $v^+:=f(e^+)$, $v^-:=f(e^-)$, we complete the proof.
\end{proof}

Next we consider an expansion of a given $(0, 0, k)$-system. Note that, for a subspace $W$ of $V$, one has $V=W \oplus W^\perp$ if $\hoge< , >$ is nondegenerate on $W$. Moreover if $\sign V=(p, q, r)$ and $\sign W=(s, t, 0)$, then we have $\sign W^\perp=(p-s, q-t, r)$.

\begin{Prop}
\label{prop : extension of basis of sign=(0, 0, k)}
Let $(p, q, 0):=\sign (V, \hoge< , >)$, and $\{w_1, \ldots, w_k\}$ be its $(0, 0, k)$-system with $k \in \Z_{\geq 1}$. Then there exists a $(p, q, 0)$-basis $\{x_1, \ldots, x_p, y_1, \ldots, y_q\}$ of $V$ such that 
\begin{align*}
w_i=x_i+y_i \quad (i \in \{1, \ldots, k\}).
\end{align*}
\end{Prop}

\begin{proof}
We put
\begin{align*}
W_i:=\s\{w_1, \ldots, w_{i-1}, w_{i+1}, \ldots, w_k\}^\perp \quad (i \in \{1, \ldots, k\}).
\end{align*}
First of all we prove that, for any $i \in \{1, \ldots, k\}$, there exists a $(1, 1, 0)$-system $\{x_i, y_i\}$ of $V$ such that 
\begin{align}
\label{condition : (1, 1)-p.o.s and decomposition}
\{x_i, y_i\} \subset W_i, \quad w_i=x_i+y_i.
\end{align}
Take an arbitrary $i \in \{1, \ldots, k\}$. Then one has
\begin{align*}
w_i \in {W_i}.
\end{align*}
Since $V$ is nondegenerate, we have
\begin{align*}
\s\{w_1, \ldots, w_{i-1}, w_{i+1}, \ldots, w_k\}=(\s\{w_1, \ldots, w_{i-1}, w_{i+1}, \ldots, w_k\}^\perp)^\perp={W_i}^\perp.
\end{align*}
If $w_i \in \rad W_i$, then 
\begin{align*}
w_i \in {W_i}^\perp=\s\{w_1, \ldots, w_{i-1}, w_{i+1}, \ldots, w_k\},
\end{align*}
however, this is a contradiction since $w_1, \ldots, w_k$ are linearly independent. Hence $w_i \in W_i \setminus \rad W_i$, and from Lemma~\ref{prop : decompose to space and time}, the assertion (\ref{condition : (1, 1)-p.o.s and decomposition}) holds.

Since $\hoge< , >$ is nondegenerate on $V_i:=\s\{x_i, y_i\}$, one has $V=V_i \oplus {V_i}^\perp$. Note that ${V_i}^\perp$ is nondegenerate. Then we have
\begin{align*}
\s\{w_1, \ldots, w_{i-1}, w_{i+1}, \ldots, w_k\} \subset {V_i}^\perp,
\end{align*}
hence we can repeat the same procedure as the argument of (\ref{condition : (1, 1)-p.o.s and decomposition}). Thus there exists a $(k, k, 0)$-system $\{x_1, \ldots, x_k, y_1, \ldots, y_k\}$ of $V$. Therefore we put
\begin{align*}
\widetilde{W}:=\s\{x_1, \ldots, x_k, y_1, \ldots, y_k\},
\end{align*}
and one has $V=\widetilde{W} \oplus {\widetilde{W}}^\perp$. Since $V$ and $\widetilde{W}$ are nondegenerate, so is ${\widetilde{W}}^\perp$, and its signature is given by
\begin{align*}
\sign {\widetilde{W}}^\perp=(p-k, q-k, 0).
\end{align*}
Thus by Proposition~\ref{prop : existence of (p, q, r)-basis}, there exists a $(p-k, q-k, 0)$-basis 
\begin{align*}
\{x_{k+1}, \ldots, x_p, y_{k+1}, \ldots, y_q\}
\end{align*}
of ${\widetilde{W}}^\perp$. Hence $V$ has the desired $(p, q, 0)$-basis, which completes the proof.
\end{proof}

By Proposition~\ref{prop : extension of basis of sign=(0, 0, k)}, one can construct a $(p, q, r)$-basis of $V$ from a given $(s, t, u)$-basis of its subspace.

\begin{Prop}
\label{prop : extension of a given basis with the weakest assumption and the strongest result}
Let $(p, q, r):=\sign (V, \hoge< , >)$ and $W$ be a subspace of $V$ such that
\begin{align*}
\sign (W, \hoge< , >)=(s, t, u), \quad \dim(W \cap \rad (V, \hoge< , >))=k.
\end{align*}
Fix an $(s, t, u)$-basis 
\begin{align*}
\{x_1, \ldots, x_s, y_1, \ldots, y_t, z_1, \ldots, z_u\} \quad (z_{u-k+1}, \ldots, z_u \in \rad V)
\end{align*}
of $W$. Then $V$ has a $(p, q, r)$-basis 
\begin{align*}
\{\alpha_1, \ldots, \alpha_p, \beta_1, \ldots, \beta_q, \gamma_1, \ldots, \gamma_r\}
\end{align*}
such that  
\begin{align*}
&x_i=\alpha_i \quad (i \in \{1, \ldots, s\}),\\
&y_i=\beta_i \quad (i \in \{1, \ldots, t\}),\\
&z_i=\alpha_{s+i}+\beta_{t+i} \quad (i \in \{1, \ldots, u-k\}),\\
&z_{u-k+i}=\gamma_i \quad (i \in \{1, \ldots, k\}).
\end{align*}
\end{Prop}

\begin{proof}
First of all, we put
\begin{align*}
&W^\pm:=\s\{x_1, \ldots, x_s, y_1, \ldots, y_t\},\\
&W_0:=\s\{z_1, \ldots, z_{u-k}\},\\
&W_\nul:=\s\{z_{u-k+1}, \ldots, z_u\}.
\end{align*}
By the assumption, it satisfies $(W^\pm \oplus W_0) \cap \rad V=\{0\}$. Therefore there exists a subspace $U$ of $V$ such that
\begin{align*}
V=U \oplus \rad V, \quad W^\pm \oplus W_0 \subset U.
\end{align*}
Note that $U$ is nondegenerate. Here we define 
\begin{align*}
(W^\pm)^\perp_U:=\{u \in U \mid \forall w \in W^\pm,\ \hoge<u, w>=0\}.
\end{align*}
Since $W^\pm$ is a nondegenerate subspace of $U$, one has $U=W^\pm \oplus (W^\pm)^\perp_U$. Hence we have 
\begin{align*}
V=U \oplus \rad V=W^\pm \oplus (W^\pm)^\perp_U \oplus \rad V.
\end{align*}
Remember that $W_0 \subset (W^\pm)^\perp_U$ and $W_\nul \subset \rad V$. We will construct the bases of $W^\pm$, $(W^\pm)^\perp_U$, and $\rad V$.

Regarding the basis $\{x_1, \ldots, x_s, y_1, \ldots, y_t\}$ of $W^\pm$, we put 
\begin{align}
\label{eq : construction of alpha 1 to s and beta 1 to t}
\alpha_i:=x_i \quad (i \in \{1, \ldots, s\}), \quad \beta_i:=y_i \quad (i \in \{1, \ldots, t\}).
\end{align}

Next we construct a $(p-s, q-t, 0)$-basis of $(W^\pm)^\perp_U$. Recall that $U$ and $W^\pm$ are nondegenerate. Hence $(W^\pm)^\perp_U$ is nondegenerate, and its signature is given by
\begin{align*}
\sign (W^\pm)^\perp_U=(p-s, q-t, 0).
\end{align*}
Since $\{z_1, \ldots, z_{u-k}\}$ is a $(0, 0, u-k)$-system of $W_0$, by Proposition~\ref{prop : extension of basis of sign=(0, 0, k)}, there exists a $(p-s, q-t, 0)$-basis $\{\alpha_{s+1}, \ldots, \alpha_p, \beta_{t+1}, \ldots, \beta_q\}$ of $(W^\pm)^\perp_U$ such that
\begin{align}
\label{eq : z 1 to u-k}
z_i=\alpha_{s+i}+\beta_{t+i} \quad (i \in \{1, \ldots, u-k\}).
\end{align}

Finally we construct a $(0, 0, r)$-basis of $\rad V$. Since $\{z_{u-k+1}, \ldots, z_u\}$ is a basis of $W_\nul$ and $W_\nul \subset \rad V$, there exists a basis $\{\gamma_1, \ldots, \gamma_r\}$ of $\rad V$ such that
\begin{align}
\label{eq : construction of gamma 1 to k}
\gamma_i=z_{u-k+i} \quad (i \in \{1, \ldots, k\}).
\end{align}
From (\ref{eq : construction of alpha 1 to s and beta 1 to t}), (\ref{eq : z 1 to u-k}) and (\ref{eq : construction of gamma 1 to k}), one obtains the desired $(p, q, r)$-basis of $V$, which completes the proof.
\end{proof}

\subsection{The proof of Proposition~\ref{prop : flag mfd iff signature and radical}}
\label{sub2}

In this subsection, we prove Proposition~\ref{prop : flag mfd iff signature and radical}. First of all, we show that one can extend a given linear isometry between subspaces to the entire nondegenerate space.

\begin{Prop}
\label{prop : extension of isometry on W to an isometry on V}
Let $V$ be a nondegenerate space, and $W_1$ and $W_2$ be subspaces of $V$ with $\sign (W_1, \hoge< , >)=\sign (W_2, \hoge< , >)$. Then for any linear isometry $f: W_1 \to W_2$, there exists a linear isometry $\widetilde{f}: V \to V$ such that $\widetilde{f}\mid_{W_1}=f$.
\end{Prop}

\begin{proof}
Let $(s, t, u):=\sign W_1=\sign W_2$. Take an arbitrary linear isometry $f: W_1 \to W_2$. Here we fix an $(s, t, u)$-basis 
\begin{align*}
\{x_1, \ldots, x_s, y_1, \ldots, y_t, z_1, \ldots, z_u\}
\end{align*}
of $W_1$. Since $f: W_1 \to W_2$ is a linear isometry, 
\begin{align*}
\{f(x_1), \ldots, f(x_s), f(y_1), \ldots, f(y_t), f(z_1), \ldots, f(z_u)\}
\end{align*}
is an $(s, t, u)$-basis of $W_2$. Note that
\begin{align*}
\dim (W_1 \cap \rad V)=\dim (W_2 \cap \rad V)=0,
\end{align*}
since $V$ is nondegenerate. Then from Proposition~\ref{prop : extension of a given basis with the weakest assumption and the strongest result}, there exist two $(p, q, 0)$-bases 
\begin{align*}
\{\alpha_1, \ldots, \alpha_p, \beta_1, \ldots, \beta_q\}, \quad \{\alpha'_1, \ldots, \alpha'_p, \beta'_1, \ldots, \beta'_q\}
\end{align*}
of $V$ such that
\begin{align}
\label{eq : alpha for extension of linear isometry}
&x_i=\alpha_i, \quad f(x_i)=\alpha'_i \quad (i \in \{1, \ldots, s\}),\\
\label{eq : beta for extension of linear isometry}
&y_i=\beta_i, \quad f(y_i)=\beta'_i \quad (i \in \{1, \ldots, t\}),\\
\label{eq : z and eta for extension of linear isometry}
&z_i=\alpha_{s+i}+\beta_{t+i}, \quad f(z_i)=\alpha'_{s+i}+\beta'_{t+i} \quad (i \in \{1, \ldots, u\}).
\end{align}
Here we define $\widetilde{f}: V \to V$ by mapping the former basis to the latter, that is,
\begin{align*}
\widetilde{f}(\alpha_i):=\alpha'_i \quad (i \in \{1, \ldots, p\}), \quad \widetilde{f}(\beta_i):=\beta'_i \quad (i \in \{1, \ldots, q\}).
\end{align*}
One can easily check that $\widetilde{f}: V \to V$ is a linear isometry such that $\widetilde{f}\mid_{W_1}=f$ from (\ref{eq : alpha for extension of linear isometry}), (\ref{eq : beta for extension of linear isometry}) and (\ref{eq : z and eta for extension of linear isometry}).
\end{proof}

The next lemma follows from basic linear algebra.

\begin{Lem}
\label{lem : f(rad W)=rad f(W)}
Let $W$ be a subspace of $V$, and $f: V \to V$ be a linear isometry with respect to $\hoge< , >$. Then one has $f(\rad (W, \hoge< , >))=\rad (f(W), \hoge< , >)$.
\end{Lem}

Next we show an equivalent condition for the classification of subspaces by linear isometries.


\begin{Prop}
\label{prop : generalization of O(p, q, r) on G_k(R^{p+q+r})}
For any two subspaces $U$ and $W$ of $V$, the following two conditions are equivalent.
\begin{itemize}
\item[(1)] There exists a linear isometry $f: V \to V$ with respect to $\hoge< , >$ such that $U=f(W)$.
\item[(2)] All of the following hold.
	\begin{itemize}
	\item[(i)] $\sign (U, \hoge< , >)=\sign (W, \hoge< , >)$.
	\item[(ii)] $\dim (U \cap \rad (V, \hoge< , >))=\dim (W \cap \rad (V, \hoge< , >))$.
	\end{itemize}
\end{itemize}
\end{Prop}

\begin{proof}
First we assume (1), and show (2). Let $(s, t, u):=\sign W$ with respect to $\hoge< , >$. Then by Proposition~\ref{prop : existence of (p, q, r)-basis}, there exists an $(s, t, u)$-basis
\begin{align*}
\{x_1, \ldots, x_s, y_1, \ldots, y_t, z_1, \ldots, z_u\}
\end{align*}
of $W$. Since $f\mid_W: W \to U$ is a linear isometry, 
\begin{align*}
\{f(x_1), \ldots, f(x_s), f(y_1), \ldots, f(y_t), f(z_1), \ldots, f(z_u)\}
\end{align*}
is an $(s, t, u)$-basis of $U$. Hence one has $\sign U=\sign W$, which proves (i). Regarding the assertion (ii), by Lemma~\ref{lem : f(rad W)=rad f(W)} we have
\begin{align*}
\rad V=\rad f(V)=f(\rad V),
\end{align*}
thus one has
\begin{align*}
U \cap \rad V=f(W) \cap f(\rad V)=f(W \cap \rad V).
\end{align*}
This completes the proof of (ii).

Next let us assume (2), and we show (1). Put 
\begin{align*}
&(p, q, r):=\sign V,\\
&(s, t, u):=\sign U=\sign W
\end{align*}
We fix $(s, t, u)$-bases of $W$ and $U$ which satisfy the assumption of Proposition~\ref{prop : extension of a given basis with the weakest assumption and the strongest result}. Then they can be extended to $(p, q, r)$-bases
\begin{align*}
&\{\alpha_1, \ldots, \alpha_p, \beta_1, \ldots, \beta_q, \gamma_1, \ldots, \gamma_r\},\\
&\{\alpha'_1, \ldots, \alpha'_p, \beta'_1, \ldots, \beta'_q, \gamma'_1, \ldots, \gamma'_r\}
\end{align*}
of $V$ in the way of Proposition~\ref{prop : extension of a given basis with the weakest assumption and the strongest result}. Let $f: V \to V$ be the linear isometry which maps the former basis to the latter. Then we have $U=f(W)$, which completes the proof.
\end{proof}

Finally we prove Proposition~\ref{prop : flag mfd iff signature and radical} by using Propositions~\ref{prop : extension of isometry on W to an isometry on V} and \ref{prop : generalization of O(p, q, r) on G_k(R^{p+q+r})}.

\begin{proof}[Proof of Proposition~\ref{prop : flag mfd iff signature and radical}]
Take arbitrary $(V_{k_1}, V_{k_2}), (W_{k_1}, W_{k_2}) \in F_{k_1, k_2}$. First of all, we assume (1). Then there exists $g \in \OO(p, q)$ such that
\begin{align*}
(V_{k_1}, V_{k_2})=g.(W_{k_1}, W_{k_2})=(gW_{k_1}, gW_{k_2}).
\end{align*}
Under this assumption, we show (2), that is, we prove the following:
\begin{itemize}
\item[(i)] $\sign (V_{k_2}, \hoge< , >_0)=\sign (W_{k_2}, \hoge< , >_0)$.
\item[(ii)] $\sign (V_{k_1}, \hoge< , >_0)=\sign (W_{k_1}, \hoge< , >_0)$.
\item[(iii)] $\dim (V_{k_1} \cap \rad (V_{k_2}, \hoge< , >_0))=\dim (W_{k_1} \cap \rad (W_{k_2}, \hoge< , >_0))$.
\end{itemize}
The assertions (i) and (ii) follow from Proposition~\ref{prop : generalization of O(p, q, r) on G_k(R^{p+q+r})}, and (iii) holds from Lemma~\ref{lem : f(rad W)=rad f(W)}.

Next we assume (2) and show (1). Since $\sign V_{k_2}=\sign W_{k_2}$ and $\rad \R^{p+q}=\{0\}$, from Proposition~\ref{prop : generalization of O(p, q, r) on G_k(R^{p+q+r})}, there exists a linear isometry $f: \R^{p+q} \to \R^{p+q}$ such that 
\begin{align*}
V_{k_2}=f(W_{k_2}). 
\end{align*}
We then find a linear isometry mapping $f(W_{k_1})$ to $V_{k_1}$. From the assumption (ii), one has
\begin{align}
\label{eq : sign V_{k_1}=sign f(W_{k_1})}
\sign V_{k_1}=\sign W_{k_1}=\sign f(W_{k_1}).
\end{align}
Moreover by $V_{k_2}=f(W_{k_2})$, we have
\begin{align*}
f(W_{k_1}) \cap f(\rad W_{k_2})=f(W_{k_1}) \cap \rad f(W_{k_2})=f(W_{k_1}) \cap \rad V_{k_2}.
\end{align*}
Hence by the assumption (iii), we obtain
\begin{align}
\label{eq : equation between dimension of intersection}
\dim(V_{k_1} \cap \rad V_{k_2})&=\dim(W_{k_1} \cap \rad W_{k_2})=\dim(f(W_{k_1}) \cap \rad V_{k_2}).
\end{align}
Therefore by (\ref{eq : sign V_{k_1}=sign f(W_{k_1})}), (\ref{eq : equation between dimension of intersection}) and Proposition~\ref{prop : generalization of O(p, q, r) on G_k(R^{p+q+r})}, there exists a linear isometry $h: V_{k_2} \to V_{k_2}$ such that
\begin{align*}
V_{k_1}=h(f(W_{k_1})).
\end{align*}
From Proposition~\ref{prop : extension of isometry on W to an isometry on V}, there exists a linear isometry $\widetilde{h}: \R^{p+q} \to \R^{p+q}$ such that
\begin{align*}
\widetilde{h}\mid_{V_{k_2}}=h.
\end{align*}
Hence from the above argument, we have
\begin{align*}
V_{k_1}=(\widetilde{h} \circ f)(W_{k_1}), \quad V_{k_2}=(\widetilde{h} \circ f)(W_{k_2}).
\end{align*}
Since $\widetilde{h} \circ f: \R^{p+q} \to \R^{p+q}$ is a linear isometry with respect to $\hoge< , >_0$, which completes the proof.
\end{proof}

\subsection{The proofs of Propositions~\ref{prop : signature on V_{p+q-2}} and \ref{prop : sign_V W}}
\label{sub3}

In this subsection, we prove Propositions~\ref{prop : signature on V_{p+q-2}} and \ref{prop : sign_V W}. First, we prove Proposition~\ref{prop : signature on V_{p+q-2}}. Recall that $A$ is the set of all possible signatures $\sign (V, \hoge< , >_0)$ of codimension-two subspaces $V$ of $\R^{p+q}$.

\begin{proof}[Proof of Proposition~\ref{prop : signature on V_{p+q-2}}]
First of all, we show that 
\begin{align}
\label{A is included in the RHS}
A \subset \left\{
\begin{array}{lll}
	(p-2, q, 0), &(p-1, q-1, 0), &(p, q-2, 0),\\
	(p-2, q-1, 1), &(p-1, q-2, 1), &(p-2, q-2, 2)
\end{array}\right\} \cap (\Z_{\geq 0})^3.
\end{align}
Take an arbitrary subspace $V$ of $\R^{p+q}$ with $\dim V=p+q-2$, and we put $\sign V=(s, t, u)$, where $s, t, u \in \Z_{\geq 0}$. Then we have
\begin{align}
\label{eq : s+t+u=p+q-2}
s+t+u=p+q-2.
\end{align}
Since $\hoge< , >_0$ is nondegenerate on $\R^{p+q}$, one has by Proposition~\ref{prop : extension of a given basis with the weakest assumption and the strongest result} that
\begin{align}
\label{ineq : s+u<p, t+u<q}
s+u \leq p, \quad t+u \leq q.
\end{align}
By (\ref{eq : s+t+u=p+q-2}) and (\ref{ineq : s+u<p, t+u<q}), we obtain
\begin{align}
\label{ineq : 0<=u<=2}
0 \leq u \leq 2. 
\end{align}
In order to calculate $\sign V$, we have only to enumerate all possible integers $s, t, u \in \Z_{\geq 0}$ satisfying the conditions (\ref{eq : s+t+u=p+q-2}), (\ref{ineq : s+u<p, t+u<q}) and (\ref{ineq : 0<=u<=2}). 

Let us fix $u=0$. By (\ref{eq : s+t+u=p+q-2}) and (\ref{ineq : s+u<p, t+u<q}), we have
\begin{align*}
s+t=p+q-2, \quad 0 \leq s \leq p, \quad 0 \leq t \leq q.
\end{align*}
According to these conditions, we have
\begin{align*}
(s, t) \in \{(p-2, q),\ (p-1, q-1),\ (p, q-2)\} \cap (\Z_{\geq 0})^2.
\end{align*}
For other two cases of $u$, one can summarize as follows:
\begin{itemize}
	\setlength{\itemsep}{3pt} 
\item if $u=1$, then $(s, t) \in \{(p-2, q-1), (p-1, q-2)\} \cap (\Z_{\geq 0})^2$,
\item if $u=2$, then $(s, t) \in \{(p-2, q-2)\} \cap (\Z_{\geq 0})^2$.
\end{itemize}
Therefore by the above arguments, we obtain (\ref{A is included in the RHS}).

One can prove the converse inclusion by constructing subspaces $V$ with prescribed signatures. In fact, by Proposition~\ref{prop : existence of (p, q, r)-basis}, there exists a $(p, q, 0)$-basis $\{x_1, \ldots, x_p, y_1, \ldots, y_q\}$ of $\R^{p+q}$ with respect to $\hoge< , >_0$. Hence, a subspace 
\begin{align*}
V:=\s\{x_1, \ldots, x_{p-2}, y_1, \ldots, y_{q-2}, x_{p-1}+y_{q-1}, x_p+y_q\}
\end{align*}
satisfies $\sign V=(p-2, q-2, 2)$. We can similarly construct subspaces $V$ for the other five triplets, which completes the proof.
\end{proof}

Finally, we prove Proposition~\ref{prop : sign_V W}. Recall that $B$ is the set of all possible signatures $\sign_V (W, \hoge< , >_0)$ of one-dimensional subspaces $W$ of $V$.

\begin{proof}[Proof of Proposition~\ref{prop : sign_V W}]
First of all, we show the first assertion. Since $\sign V=(s, t, u)$, by Proposition~\ref{prop : existence of (p, q, r)-basis}, there exists an $(s, t, u)$-basis
\begin{align*}
\{x_1, \ldots, x_s, y_1, \ldots, y_t, z_1, \ldots, z_u\}
\end{align*}
of $V$. Take an arbitrary $v \in V$. In terms of this basis, it can be expressed as
\begin{align}
\label{eq : basis expression of V}
v=\sum_{i=1}^s a_i x_i+\sum_{j=1}^t b_j y_j+\sum_{k=1}^u c_k z_k,
\end{align}
where $a_1, \ldots, a_s, b_1, \ldots, b_t, c_1, \ldots, c_u \in \R$. Then one has
\begin{align}
\label{eq : norm of v with respect to a and b}
\hoge<v, v>_0=\sum_{i=1}^s {a_i}^2-\sum_{j=1}^t {b_j}^2.
\end{align}
Therefore it is easy to verify the first assertion, that is, $(1, 0, 0) \in B$ if and only if $V$ has a non-zero space-like vector, which is equivalent to $s \geq 1$ by (\ref{eq : basis expression of V}) and (\ref{eq : norm of v with respect to a and b}). We can similarly show the second and the fourth assertions. Regarding the third assertion, $(0, 0, 1) \in B$ if and only if $V$ has a light-like vector $v \notin \rad V$, which is equivalent to $s, t \geq 1$ by (\ref{eq : norm of v with respect to a and b}). This completes the proof.
\end{proof}


\begin{thebibliography}{99} 




\bibitem{Barnet}
F. Barnet, \textit{On Lie groups that admit left-invariant Lorentz metrics of constant sectional curvature}, Illinois J. Math. \textbf{33} (1989), no. 4, 631--642.




\bibitem{BL}
C. B\"{o}hm and R. A. Lafuente, \textit{Non-compact Einstein manifolds with symmetry}, arXiv:2107.04210v1.



\bibitem{BT}
M. Boucetta and O. Tibssirte, \textit{On Einstein Lorentzian nilpotent Lie groups}, J. Pure Appl. Algebra \textbf{224} (2020), no. 12, 106443, 22 pp.





\bibitem{Cao}
H.-D. Cao, \textit{Recent progress on Ricci solitons}, Adv. Lect. Math. \textbf{11} (2010), 1--38.


\bibitem{CR}
D. Conti and F. A. Rossi, \textit{Ricci-flat and Einstein pseudoriemannian nilmanifolds}, Complex Manifolds \textbf{6} (2019), no. 1, 170--193.


\bibitem{CP}
L. A. Cordero and P. E. Parker, \textit{Left-invariant Lorentzian metrics on 3-dimensional Lie groups}, Rend. Mat. Serie VII \textbf{17} (1997), 129--155.

\bibitem{DB}
K. L. Duggal and A. Bejancu, \textit{Lightlike submanifolds of semi-Riemannian manifolds and applications}, Mathematics and its Applications, \textbf{364}. Kluwer Academic Publishers Group, Dordrecht, (1996). 

\bibitem{G}
M. Guediri, \textit{Sur la compl\'{e}tude des pseudo-m\'{e}triques invariantes a gauche sur les groupes de Lie nilpotens}, Rend. Sem. Mat. Univ. Politec. Torino \textbf{52} (1994), no. 4, 371--376.




\bibitem{GB}
M. Guediri and M. Bin-Asfour, \textit{Ricci-flat left-invariant Lorentzian metrics on 2-step nilpotent Lie groups}, Arch. Math. (Brno) \textbf{50} (2014), no. 3, 171--192.


\bibitem{HL}
K. Y. Ha and J. B. Lee, \textit{Left invariant metrics and curvatures on simply connected three-dimensional Lie groups}, Math. Nachr. \textbf{282} (2009), no. 6, 868--898. 



\bibitem{HT}
T. Hashinaga and H. Tamaru, \textit{Three-dimensional solvsolitons and the minimality of the corresponding submanifolds}, Internat. J. Math. \textbf{28} (2017), 1750048 (31 pages).  





\bibitem{KOTT}
A. Kubo, K. Onda, Y. Taketomi and H. Tamaru, \textit{On the moduli spaces of left-invariant pseudo-Riemannian metrics on Lie groups}, Hiroshima Math. J. \textbf{46} (2016), 357--374.

\bibitem{KT}
Y. Kondo and H. Tamaru, \textit{A classification of left-invariant Lorentzian metrics on some nilpotent Lie groups}, Tohoku Math. J. to appear, arXiv:2011.09118v1.  



\bibitem{KTT} 
H. Kodama, A. Takahara and H. Tamaru, \textit{The space of left-invariant metrics on a Lie group up to isometry and scaling}, Manuscripta Math. \textbf{135} (2011), 229--243. 


\bibitem{Lauret4}
J. Lauret, \textit{Ricci soliton homogeneous nilmanifolds}, Math. Ann. \textbf{319} (2001), no. 4, 715--733.


\bibitem{Lauret}
J. Lauret, \textit{Degenerations of Lie algebras and geometry of Lie groups}, Differential Geom. Appl. \textbf{18} (2003), no. 2, 177--194. 


\bibitem{Lauret5}
J. Lauret, \textit{Ricci soliton solvmanifolds}, J. Reine Angew. Math. \textbf{650} (2011), 1--21.







\bibitem{Matsuki}
T. Matsuki, \textit{Orbits on affine symmetric spaces under the action of parabolic subgroups}, Hiroshima Math. J. \textbf{12} (1982), 307--320. 


\bibitem{Matsuki1}
T. Matsuki, \textit{The orbits of affine symmetric spaces under the action of minimal parabolic subgroups}, J. Math. Soc. Japan \textbf{31} (1979), no. 2, 331--357. 



\bibitem{Milnor}
J. Milnor, \textit{Curvatures of left invariant metrics on Lie groups}, Advances in Math. \textbf{21} (1976), no. 3, 293--329.

\bibitem{Nomizu}
K. Nomizu, \textit{Left-invariant Lorentz metrics on Lie groups}, Osaka J. Math. \textbf{16} (1979), 143--150.


\bibitem{Onda}
K. Onda, \textit{Examples of algebraic Ricci solitons in the pseudo-Riemannian case}, Acta Math. Hungar. \textbf{144} (2014), no. 1, 247--265. 


\bibitem{O'Neill}
B. O'Neill, \textit{Semi-Riemannian geometry with applications to relativity}, Pure and Applied Mathematics, \textbf{103}. Academic Press, Inc. [Harcourt Brace Jovanovich, Publishers], New York, (1983). 


\bibitem{Rahmani}
S. Rahmani, \textit{M\'etriques de Lorentz sur les groupes de Lie unimodulaires, de dimension trois}, J. Geom. Phys. \textbf{9} (1992), no. 3, 295--302.

\bibitem{RR}
N. Rahmani and S. Rahmani, \textit{Lorentzian geometry of the Heisenberg group}, Geom. Dedicata \textbf{118} (2006), 133--140.


\bibitem{Topping}
P. Topping, \textit{Lectures on the Ricci flow}, London Mathematical Society Lecture Note Series, vol. 325, Cambridge University Press, Cambridge, 2006.


\bibitem{Will}
C. Will, \textit{The space of solvsolitons in low dimensions}, Ann. Glob. Anal. Geom. \textbf{40} (2011), no. 3, 291--309.


\bibitem{Wolf}
J. A. Wolf, \textit{Finiteness of orbit structure for real flag manifolds}, Geometriae Dedicata \textbf{3} (1974), 377--384.




























\end{thebibliography}
\end{document}